\RequirePackage{fix-cm}
\documentclass[smallextended]{svjour3}       % onecolumn (second format)
\smartqed  % flush right qed marks, e.g. at end of proof
\usepackage{amsmath, amssymb, bm, mathrsfs, mathtools}
\usepackage{cite, color, graphicx}
\usepackage[colorlinks=true,linkcolor=blue]{hyperref}
\mathtoolsset{showonlyrefs=true,showmanualtags=true}

\spnewtheorem{theorem}{Theorem}[section]{\bfseries}{\itshape}

\spnewtheorem{lemma}[theorem]{Lemma}{\bfseries}{\itshape}
\spnewtheorem{corollary}[theorem]{Corollary}{\bfseries}{\itshape}
\spnewtheorem{problem}[theorem]{Problem}{\bfseries}{\itshape}
\spnewtheorem{definition}[theorem]{Definition}{\bfseries}{}
\spnewtheorem{proposition}[theorem]{Proposition}{\bfseries}{\itshape}
\spnewtheorem{example}[theorem]{Example}{\bfseries}{}

\spnewtheorem{remark}[theorem]{Remark}{\bfseries}{\upshape}
\spnewtheorem{assumption}[theorem]{Assumption}{\bfseries}{\itshape}

\newcommand{\im}{\mathop{\rm Im}\nolimits}
\newcommand{\ran}{\mathop{\rm ran}\nolimits}
\newcommand{\re}{\mathop{\rm Re}\nolimits}

\makeatletter
\renewcommand{\theenumi}{\rm{\@alph\c@enumi)}}

\renewcommand{\theenumii}{\rm{(\@roman\c@enumii)}}

\makeatother
\makeatletter

\@addtoreset{equation}{section}
\makeatother

\begin{document}
	
	\title{The Cayley transform of
		the generator of a polynomially stable $C_0$-semigroup
		\thanks{This work was supported by JSPS KAKENHI Grant Number JP20K14362.}
	}
	
	
	\author{Masashi Wakaiki
	}
	
	
	\institute{M. Wakaiki \at
		Graduate School of System Informatics, Kobe University, Nada, Kobe, Hyogo 657-8501, Japan \\
		Tel.: +8178-803-6232\\
		Fax: +8178-803-6392\\
		\email{wakaiki@ruby.kobe-u.ac.jp}           %  \\
	}
	
	\date{Received: date / Accepted: date}

	\maketitle
	
	\begin{abstract}
		In this paper, we study the decay rate of the Cayley transform
		of the generator of a polynomially stable $C_0$-semigroup.
		To estimate the decay rate of the Cayley transform,
		we develop an integral condition on resolvents
		for polynomial stability. Using this integral condition,
		we relate polynomial stability to Lyapunov equations.
		We also study robustness of polynomial stability
		for a certain class of structured perturbations.
	\end{abstract}
	
	\section{Introduction}
	Consider a $C_0$-semigroup $(T(t))_{t\geq 0}$ on a Hilbert space with
	generator $A$, and suppose that $(T(t))_{t\geq 0}$ is polynomially 
	stable with parameter $\alpha>0$, that is,
	$(T(t))_{t\geq 0}$ is uniformly bounded, 
	the spectrum of $A$ is contained in the open left half-plane,
	and there exists $M>0$ such that for all $t>0$,
	\[
	\|T(t)A^{-1}\| \leq 
	\frac{M}{t^{1/\alpha}}.
	\]
	We consider the following question:
	Is polynomial decay of $(T(t))_{t\geq 0}$ passed to
	the Cayley transform $A_d := (I+A)(I-A)^{-1}$?
	The quantitive behavior of the operator norm $\|T(t)A^{-1}\|$
	has been extensively studied; see e.g., \cite{Liu2005PDR, 
		Batkai2006,Batty2008, Borichev2010,Batty2016,Rozendaal2019}. 
	A discrete analogue, the quantified 
	Katznelson-Tzafriri theorem, has also been investigated in \cite{
		Seifert2015, Seifert2016, Cohen2016, Ng2020}.
	However, to the author's knowledge, 
	it has not been well established
	whether and how 
	polynomial decay of $(T(t))_{t\geq 0}$ in the continuous setting
	yields the decay of 
	the corresponding Cayley transform in the discrete setting.
	The purpose of this paper is to show that 
	polynomial decay of a $C_0$-semigroup
	is preserved under the Cayley transformation in a certain sense.
	
	Applications of the Cayley transform of a semigroup 
	generator arise in numerical analysis \cite{Piskarev2007} and 
	system theory \cite[Section~12.3]{Staffans2005}.
	In the finite-dimensional case, a matrix
	and its 
	Cayley transform 
	share the same stability properties, but
	this does not hold in the infinite-dimensional case.
	In fact, in the Banach space setting,
	the Cayley transform of the generator of 
	even an exponentially stable $C_0$-semigroup
	may not be power bounded \cite[Lemma~2.1]{Gomilko2011}.
	For the case of Hilbert spaces,
	it is still unknown whether 
	the corresponding Cayley transform is power bounded for
	every generator of a uniformly bounded
	$C_0$-semigroup, as mentioned in Section~5.5 of \cite{Batty2021}.
	However, some sufficient conditions for Cayley transforms
	to be power bounded have been obtained; see, e.g., 
	\cite{Gomilko2004, Guo2006, Gomilko2007, Gomilko2011}. In particular,
	it is well known that $A$ generates a $C_0$-semigroup
	of contractions on a Hilbert space if and only if the corresponding
	Cayley transform is a contraction \cite[Theorem III.8.1]{Nagy1970}.
	We refer the reader to the survey \cite{Gomilko2017} for more details.
	
	In this paper, we prove that if $(T(t))_{t\geq 0}$ is a polynomially
	stable $C_0$-semigroup with parameter $\alpha>0$ 
	on a Hilbert space with generator $A$
	such that its Cayley transform $A_d$ is power bounded, 
	then there exists $M_d>0$ such that for all $n \in \mathbb{N}$,
	\[
	\|A_d^n A^{-1}\| \leq 
	M_d \left(\frac{\log n}{n} \right)^{\frac{1}{\alpha+2}}.
	\]
	We also show that 
	in some cases, such as when $A$ is normal,
	the logarithmic correction  can be omitted.
	Moreover, we give a simple example of 
	a normal operator $A$ for which the 
	decay rate $1/n^{1/(\alpha+2)}$ cannot be improved.
	
	To obtain the decay estimate of Cayley transforms,
	we extend the Lyapunov-based approach developed by
	Guo and Zwart \cite{Guo2006}.
	In \cite{Guo2006},
	uniform boundedness and strong stability of
	a $C_0$-semigroup have been characterized in terms of
	the solution of a certain Lyapunov equation.
	The integral conditions on resolvents 
	obtained in \cite{Gomilko1999,Shi2000} for uniform boundedness
	and in \cite{Tomilov2001} for strong stability play
	an important role in this Lyapunov-based approach.
	Therefore,
	we first obtain a similar integral condition 
	for polynomial stability.
	By means of this integral condition, we next relate
	polynomial stability to the Lyapunov equation used in \cite{Guo2006}.
	Finally, we estimate the decay rate of the Cayley transform, by
	using the solution of the Lyapunov equation.
	
	As another application of the  Lyapunov-based approach,
	we 
	consider the following robustness analysis of polynomial stability:
	If $A$ generates a polynomially stable semigroup, then
	does $A+rA^{-1}$ also generate a polynomially stable semigroup
	for every $r >0$?
	Robustness of polynomial stability has been studied in \cite{Paunonen2011,
		Paunonen2012SS,Paunonen2013SS,Rastogi2020}.
	In these previous studies, perturbations are not structured,
	but the norms of the perturbations are assumed to be 
	bounded in a certain sense.
	In contrast, the class of perturbations 
	we consider is limited to $\{rA^{-1}:r>0\}$, but
	we do not place any norm conditions for perturbations.
	We show that if $A$ generates a polynomially stable semigroup
	with parameter $\alpha>0$, then for every $r>0$,
	$A+rA^{-1}$ also generates a polynomially stable semigroup
	with the same parameter $\alpha$  in the case $\alpha>2$
	and with parameter $\alpha+\varepsilon$ for arbitrary small $\varepsilon>0$
	in the case $\alpha < 2$. If $\alpha = 2$, then a logarithmic factor appears
in the rate of decay.
	
	This paper is organized as follows.
	In Section~\ref{sec:Background}, we
	collect some preliminary results on polynomial stability.
	In Section~\ref{sec:Lyap}, we present an 
	integral condition on resolvents for polynomial stability
	and then connect this stability to Lyapunov equations.
	In Section~\ref{sec:Decay_CT}, we study the decay rate of the Cayley transform
	of the generator of a polynomially stable $C_0$-semigroup.
	Section~\ref{sec:robustness} contains the robustness analysis
	of polynomial stability for the class of
	perturbations $\{rA^{-1}:r>0 \}$.

	\paragraph{Notation}
	Let $\mathbb{C}_- := \{\lambda \in \mathbb{C}:\re \lambda < 0 \}$ and
	$i\mathbb{R} := \{i \eta : \eta \in \mathbb{R}\}$.
	The closure of a subset $\Omega$ of $\mathbb{C}$ is denoted by
	$\overline{\Omega}$.
	For real-valued functions $f,g$ on $J \subset \mathbb{R}$, we write
	$f(t) = O(g(t))$ as $t \to \infty$ if
	there exist $M>0$ and $t_0 \in J$ such that 
	$f(t) \leq Mg(t)$ for every $t \geq t_0$, and
	similarly,
	$f(t) = o(g(t))$ as $t \to \infty$ if
	for every $\varepsilon >0$, there exists $t_0 \in J$ such that 
	$f(t) \leq \varepsilon g(t)$ for every $t \geq t_0$.
	Let $X$ be a Banach space.
	For a linear operator $A$ on $X$, we denote by $D(A)$ and
	$\ran (A)$ the domain and the range of $A$, respectively.
	The space of bounded linear operators on $X$ is denoted by
	$\mathcal{L}(X)$.
	For a closed operator $A:D(A) \subset X \to X$,
	we denote by $\sigma(A)$ and $\varrho(A)$ the spectrum and
	the resolvent set of $A$, respectively.
	Let $\tilde \sigma(A)$ be the extended spectrum of $A$ defined by
	\[
	\tilde \sigma(A) :=
	\begin{cases}
	\sigma(A) & \text{if $A \in \mathcal{L}(X)$} \\
	\sigma(A) \cup \{\infty\} & \text{if $A \not\in \mathcal{L}(X)$}.
	\end{cases}
	\]
	For $\lambda \in \varrho(A)$, the resolvent operator is given by
	$R(\lambda,A) := (\lambda - A)^{-1}$.
	Let $H$ be a Hilbert space.
	The inner product of $H$ is denoted by $\langle\cdot, \cdot \rangle$.
	The Hilbert space adjoint for a linear operator 
	$A$ with dense domain in $H$
	is denoted by $A^*$.
	
	\section{Background on polynomially stable semigroups}
	\label{sec:Background}
	In this section, we review the definition and 
	some important properties of polynomially stable $C_0$-semigroups.
	
	\begin{definition}
		Let $\alpha >0$. A $C_0$-semigroup $(T(t))_{t\geq 0}$ 
		on a Banach space $X$
		generated by $A:D(A) \subset X \to X$ is {\em polynomially stable with
			parameter
			$\alpha$} if $(T(t))_{t\geq 0}$ is uniformly bounded, 
		$i \mathbb{R} \subset \varrho(A)$, and 
		\begin{equation}
		\label{eq:poly_decay}
		\|T(t)A^{-1}\| = 
		O\left( \frac{1}{t^{1/\alpha}}
		\right)\qquad t\to \infty.
		\end{equation}
	\end{definition}

	The spectrum of 
	the generator of any uniformly bounded semigroup
	is contained in the closed left half-plane $\overline{\mathbb{C}_-}$.
	Therefore, if $A$ generates a polynomially stable semigroup, then
	$\sigma(A) \subset \mathbb{C}_-$.

	Polynomial decay \eqref{eq:poly_decay} 
	of a $C_0$-semigroup $(T(t))_{t\geq 0}$ on a Hilbert space
	can be characterized by orbits as well; see \cite[Theorem~2.4]{Borichev2010}.
	\begin{theorem}
		\label{thm:decay_charac}
		Let $(T(t))_{t\geq 0}$ be a uniformly bounded $C_0$-semigroup
		on a Hilbert space $H$ with generator 
		$A$ such that $i \mathbb{R} \subset \varrho(A)$. 
		For a fixed 
		$\alpha >0$, \eqref{eq:poly_decay} holds
		if and only if
		\begin{equation}
		\label{eq:poly_decay_orbits}
		\|T(t)A^{-1} x\| = o\left(\dfrac{1}{t^{1/\alpha}}\right)\qquad t\to \infty
		\end{equation}
		holds for every $x \in H$.
	\end{theorem}
	
	For
	the generator $A$ of a polynomially stable $C_0$-semigroup,
	$-A$ is sectorial in the sense of \cite[Chapter~2]{Haase2006}, and hence
	the fractional powers $(-A)^{\alpha}$ are well defined for all $\alpha \in \mathbb{R}$.
	Using the moment
	inequality (see, e.g., Proposition~6.6.4 of \cite{Haase2006}), we can normalize
	the decay rate
	in \eqref{eq:poly_decay}. See 
	\cite[Proposition~3.1]{Batkai2006} for the proof.
	\begin{lemma}
		\label{lem:frac_normalize1}
		Let 
		$\alpha >0$ and $(T(t))_{t\geq 0}$ be a uniformly bounded $C_0$-semigroup
		on a Banach space with generator 
		$A$ such that $0 \in \varrho(A)$.  Then
		\[
		\|T(t)(-A)^{-\alpha}\| = O\left(\frac{1}{t} \right)\qquad t\to \infty
		\]
		if and only if
		\[
		\|T(t)(-A)^{-\alpha \gamma }\| = O\left(\frac{1}{t^\gamma} \right)\qquad t\to \infty
		\]
		for some/all $\gamma>0$.
	\end{lemma}
	
	A similar normalization result holds also for the 
	case of orbits \eqref{eq:poly_decay_orbits}.
	The proof is essentially same as that of Lemma~\ref{lem:frac_normalize1}, i.e.,
	it is a
	consequence of the moment inequality as stated in the proof of
	Theorem~2.4 of \cite{Borichev2010}.
	However, to make our presentation self-contained,
	we give a short argument.
	\begin{lemma}
		\label{lem:frac_normalize2}
		Let 
		$\alpha >0$ and $(T(t))_{t\geq 0}$ be a uniformly bounded $C_0$-semigroup
		on a Banach space $X$ with generator 
		$A$ such that $0 \in \varrho(A)$. 
		Then
		\begin{equation}
		\label{eq:alpha_trajectory}
		\|T(t)(-A)^{-\alpha}x\| = o\left(\frac{1}{t} \right) \qquad t \to \infty
		\end{equation}
		for all $x \in X$ 
		if and only if
		\begin{equation}
		\label{eq:alphagamma_trajectory}
		\|T(t)(-A)^{-\alpha \gamma } x\| = o\left(\frac{1}{t^\gamma} \right)\qquad t\to \infty
		\end{equation}
		for all $x \in X$ and some/all $\gamma>0$.
	\end{lemma}
	\begin{proof}
		Suppose that \eqref{eq:alphagamma_trajectory} holds
		for all $x \in X$ and some $\gamma >0$. Define $\delta := \alpha \gamma$. 
		Since
		\[
		\sup_{t \geq 0} \|t^\gamma T(t) (-A)^{-\delta}x\| < \infty
		\]
		for every $x \in X$, it follows from
		the uniform boundedness principle that there exists $C>0$
		such that 
		\[
		\sup_{t \geq 0} \|t^\gamma T(t) (-A)^{-\delta}\|  \leq C.
		\]
		
		Take $x \in X$,
		$\varepsilon >0$, and $k \in \mathbb{N}$.
		Let 
		\[
		0 < \varepsilon_0 < \frac{\varepsilon}{k^{k\gamma} C^{k-1}}.
		\] 
		By \eqref{eq:alphagamma_trajectory},
		there exists $t_0>0$ such that 
		\[
		\|T(t)(-A)^{-\delta} x\| \leq \frac{\varepsilon_0}{t^\gamma }
		\qquad \forall t \geq t_0.
		\]
		For every $t \geq kt_0$,
		\begin{align*}
		\|T(t) (-A)^{-k\delta}x\| &\leq 
		\|T(t/k) (-A)^{-\delta}\|^{k-1}~\! \|T(t/k) (-A)^{-\delta}x\| \\
		&\leq 
		\frac{C^{k-1}}{(t/k)^{(k-1)\gamma }} \cdot \frac{\varepsilon_0}{(t/k)^\gamma} \\
		&< \frac{\varepsilon}{t^{k\gamma}}.
		\end{align*}
		This implies that 
		for every $k \in \mathbb{N}$,
		\begin{equation}
		\label{eq:kd_est}
		\|T(t) (-A)^{-k\delta}x\| = o \left(
		\frac{1}{t^{k\gamma}}
		\right)\qquad t\to \infty.
		\end{equation}
		By the moment inequality, for 
		every $k \in \mathbb{N}$ and
		every $\vartheta \in (0,1)$, there exists a constant $L_1>0$ such that 
		\begin{align}
		\|T(t)(-A)^{-k\delta \vartheta} x\| &=
		\|
		(-A)^{k\delta (1-\vartheta)} T(t)(-A)^{-k\delta} x
		\| \notag \\
		&\leq 
		L_1\|(-A)^{k\delta } 
		T(t)(-A)^{-k\delta } x\|^{1-\vartheta} ~\!
		\|
		T(t)(-A)^{-k\delta } x\|^{\vartheta} \notag \\
		&\leq 
		L_1M^{1-\vartheta} \|T(t)(-A)^{-k\delta } x\|^{\vartheta}
		\label{eq:kdtheta_est}
		\end{align}
		for all $t \geq 0$,
		where $M := \sup_{t \geq 0} \|T(t)x\|$. 
		This and \eqref{eq:kd_est} yield
		\[
		\|T(t)(-A)^{-k\delta \vartheta} x\|  = o \left(
		\frac{1}{t^{k\gamma \vartheta}}
		\right)\qquad t\to \infty.
		\]
		Setting $\vartheta = 1/(k\gamma)$ with $k > 1/\gamma$,
		we obtain \eqref{eq:alpha_trajectory}.
		
		Suppose that \eqref{eq:alpha_trajectory} holds for all $x \in X$.
		Take $\tilde \gamma>0$ and $x \in X$.
		Substituting $\gamma=1$ into \eqref{eq:kd_est}, 
		we have that 
		for every $k \in \mathbb{N}$,
		\[
		\|T(t)(-A)^{-k\alpha} x\| = o \left(\frac{1}{t^k} \right)\qquad t\to \infty.
		\]
		As in \eqref{eq:kdtheta_est}, we see that 
		for every $k \in \mathbb{N}$
		and every $\vartheta \in (0,1)$, 
		there exists $L_2 >0$ such that 
		\[
		\|T(t)(-A)^{-k\alpha \vartheta} x\| \leq L_2 M^{1-\vartheta} 
		\|T(t)(-A)^{-k\alpha}x  \|^{\vartheta}\qquad \forall t \geq 0,
		\]
		where $M := \sup_{t \geq 0} \|T(t)x\|$.
		Setting $\vartheta = \tilde \gamma /k$ with $k > \tilde \gamma$,
		we obtain \eqref{eq:alphagamma_trajectory} with $\gamma = \tilde
		\gamma$.
		\hfill $\Box$	
	\end{proof}
	
	In Lemma~\ref{lem:frac_normalize2}, 
	we consider the global conditions on the decay of
	all orbits $\{(T(t)x)_{t\geq 0} :x \in X\}$.
	For individual orbits, a partial result holds. Since it can be obtained 
	from the moment inequality as in \eqref{eq:kdtheta_est},
	we omit the proof.
	\begin{lemma}
		\label{lem:frac_normalize3}
		Let $\alpha,\beta >0$ and
		$(T(t))_{t\geq 0}$ be a uniformly bounded $C_0$-semigroup
		on a Banach space $X$ with generator 
		$A$ such that $0 \in \varrho(A)$. 
		If $x \in X$ satisfies 
		\[
		\|T(t)(-A)^{-\alpha}x\| = o\left(\frac{1}{t^\beta} \right)\qquad t\to \infty,
		\]
		then 
		\[
		\|T(t)(-A)^{-\alpha \gamma }x\| = o\left(\frac{1}{t^{\beta \gamma}} \right)\qquad t\to \infty
		\]
		holds for all $\gamma \in (0,1)$.
	\end{lemma}

	\section{Polynomial stability and Lyapunov equation}
	\label{sec:Lyap}
	In this section, we connect polynomial stability to
	a certain Lyapunov equation. To this end,
	we first develop an integral condition on resolvents 
	for polynomial stability. 
	\begin{proposition}
		\label{prop:decay_to_integral}
		Let $(T(t))_{t\geq 0}$ be a uniformly bounded $C_0$-semigroup
		on a Hilbert space $H$ with generator 
		$A$ such that $i \mathbb{R} \subset \varrho(A)$.
		The following three assertions hold for a fixed $\alpha >0$:
		\begin{enumerate}
			\item  \label{enu:polystable1}
			$\|T(t)(-A)^{-\alpha}\| = O(1/t)$ as $t \to \infty$ if and only if	 
			\begin{align}
			\lim_{\xi \to 0+} 
			\xi^{1-2\gamma}
			\int_{-\infty}^{\infty}
			\|R(\xi+i\eta,A)(-A)^{-\alpha \gamma }x\|^2 d\eta = 0
			\end{align}
			for all $x \in H$ and
			some/all $\gamma \in (0,1/2)$.
			\item  \label{enu:polystable2}
			If $\|T(t)(-A)^{-\alpha}\| = O(1/t)$ as $t \to \infty$, then
			\begin{equation}
			\label{eq:resol_integral_cond}
			\lim_{\xi \to 0+} 
			\frac{1}{\log (1/\xi)}
			\int_{-\infty}^{\infty}
			\|R(\xi+i\eta,A)(-A)^{-\alpha/2}x\|^2 d\eta = 0
			\end{equation}
			for all $x \in H$.
			\item \label{enu:polystable3}
			If \eqref{eq:resol_integral_cond} holds for all $x \in H$,
			then
			\begin{equation}
			\label{eq:TA_o}
			\|T(t)(-A)^{-\alpha}\|  = O \left(
			\frac{\log t}{t} 
			\right)\qquad t\to \infty.
			\end{equation}
		\end{enumerate}
	\end{proposition}
	
	To prove Proposition~\ref{prop:decay_to_integral},
	we study the decay rate of an individual orbit.
	\begin{lemma}
		\label{lem:decay_to_integral_trajectory}
		Let $(T(t))_{t\geq 0}$ be a uniformly bounded $C_0$-semigroup
		on a Hilbert space $H$ with generator 
		$A$ such that $0 \in \varrho(A)$.
		The following two assertions hold for fixed $\alpha >0$ and $x \in H$:
		\begin{enumerate}
			\item  \label{enu:polystable_trajectory1}
			If
			$\|T(t)(-A)^{-\alpha} x\| = o(1/t)$ as $t \to \infty$,
			then
			\begin{align}
			\label{eq:gamma_LT_half}
			\lim_{\xi \to 0+} 
			\xi^{1-2\gamma}
			\int_{-\infty}^{\infty}
			\|R(\xi+i\eta,A)(-A)^{-\alpha \gamma }x\|^2 d\eta = 0
			\end{align}
			for all $\gamma \in (0,1/2)$ and 
			\begin{align}
			\label{eq:gamma_E_half}
			\lim_{\xi \to 0+} 
			\frac{1}{\log (1/\xi)}
			\int_{-\infty}^{\infty}
			\|R(\xi+i\eta,A)(-A)^{-\alpha/2}x\|^2 d\eta = 0.
			\end{align}
			\item  \label{enu:polystable_trajectory2}
			If \eqref{eq:gamma_LT_half} holds for some $\gamma \in (0,1/2)$, then
			\begin{align}
			\label{eq:gamma_LT_half_converse}
			\|T(t)(-A)^{-\alpha \gamma} x\| = o\left(
			\frac{1}{t^\gamma}
			\right)\qquad t \to \infty.
			\end{align}
			On the other hand,
			if \eqref{eq:gamma_E_half} holds, then
			\begin{align}
			\label{eq:gamma_E_half_converse}
			\|T(t)(-A)^{-\alpha /2} x\| = o\left(
			\sqrt{\frac{\log t}{t}}
			\right)\qquad t \to \infty.
			\end{align}
		\end{enumerate}
	\end{lemma}

	\begin{proof}
		\ref{enu:polystable_trajectory1}
		Suppose that 
		$x \in H$ satisfies
		$\|T(t)(-A)^{-\alpha} x\| = o(1/t)$ as $t \to \infty$.
		Let $\varepsilon >0$ and $\gamma \in (0,1/2]$. 
		By Lemma~\ref{lem:frac_normalize3},
		there exists $t_0 >0$ such that for all $t \geq t_0$,
		\begin{equation}
		\label{eq:orbit_eps}
		\|T(t)(-A)^{-\alpha \gamma}x\| 
		\leq \frac{\varepsilon}{t^\gamma}.
		\end{equation}
		By the Plancherel theorem,
		\[
		\frac{1}{2\pi} 
		\int^{\infty}_{-\infty}
		\|R(\xi + i\eta,A)x \|^2 d\eta = 
		\int^{\infty}_0 e^{-2\xi t} \|T(t)x\|^2 dt
		\]
		for all $\xi >0$.
		Using \eqref{eq:orbit_eps}, we obtain
		\begin{align}
		&\int^{\infty}_0 e^{-2\xi  t} \|T(t)(-A)^{-\alpha\gamma }x\|^2 dt \notag \\
		&\qquad =
		\int^{t_0}_0 e^{-2\xi  t} \|T(t)(-A)^{-\alpha\gamma }x\|^2 dt + 
		\int^\infty_{t_0}e^{-2\xi   t} \|T(t)(-A)^{-\alpha\gamma }x\|^2 dt \notag \\
		&\qquad \leq t_0M^2 \|x\|^2 + \varepsilon^2 \int_{t_0}^{\infty}
		\frac{e^{-2\xi  t}}{t^{2\gamma}} dt \label{eq:TA_int}
		\end{align}
		for all $\xi >0$,
		where $M :=   \|(-A)^{-\alpha\gamma}\|\sup_{t \geq 0} \|T(t)\|$.
		
		First we consider the case $\gamma \in (0,1/2)$. We have that 
		\[
		\int_{0}^{\infty}
		\frac{e^{-2\xi  t}}{t^{2\gamma}} dt = \frac{\Gamma(1-2\gamma)}{(2\xi)^{1-2\gamma}}
		\]
		for all $\xi >0$,
		where $\Gamma$ is the gamma function. Hence
		\eqref{eq:TA_int} yields
		\[
		\limsup_{\xi \to 0+}
		~\!\xi^{1-2\gamma}
		\int^{\infty}_0 e^{-2\xi  t} \|T(t)(-A)^{-\alpha\gamma }x\|^2 dt \leq \frac{\Gamma(1-2\gamma)\varepsilon^2}{2^{1-2\gamma}}.
		\]
		Since $\varepsilon>0$ was arbitrary, it follows that 
		\eqref{eq:gamma_LT_half} holds.
		
		Next we investigate the case $\gamma = 1/2$.
		By the inequality (5) of \cite{Gautschi1959}, 
		the exponential integral
		satisfies
		\[
		\int_\tau^{\infty} \frac{e^{-t}}{t} dt
		\leq 
		e^{-\tau}\log\left( 1+ \frac{1}{\tau} \right)
		\]
		for all $\tau >0$.
		Hence
		\[
		\int_{t_0}^{\infty}
		\frac{e^{-2\xi  t}}{t} dt =
		\int_{2\xi  t_0}^{\infty}
		\frac{e^{-t}}{t} dt \leq e^{-2\xi  t_0} \log\left ( 1 + \frac{1}{2\xi  t_0} \right)
		\]
		for every  $\xi >0$.
		Applying this estimate to \eqref{eq:TA_int}, we obtain
		\begin{equation}
		\label{eq:limsup}
		\limsup_{\xi  \to 0+}
		\frac{1}{\log\left ( 1 + \frac{1}{2\xi  t_0} \right)}
		\int^{\infty}_0 e^{-2\xi  t} \|T(t)(-A)^{-\alpha/2}x\|^2 dt 
		\leq \varepsilon^2.
		\end{equation}
		There exists  $0 < \xi_0 < 1$ such that for all $\xi \in (0,\xi_0)$,
		\[
		\log\left ( 1 + \frac{1}{2\xi  t_0} \right) \leq 2\log (1/\xi).
		\]
		Since $\varepsilon>0$ was arbitrary, it follows from \eqref{eq:limsup}
		that \eqref{eq:gamma_E_half} holds.
		
		\ref{enu:polystable_trajectory2}
		By Theorem~I.3.9 and Proposition~I.3.10 of \cite{Eisner2010}, we have that 
		\[
		T(t)x = \frac{1}{2\pi t} \int^{\infty}_{-\infty} e^{(\xi  + i \eta)t} 
		R(\xi  + i\eta,A)^2 x d\eta
		\]
		for all $x \in H$, $\xi >0$, and $t >0$.
		The Cauchy-Schwartz inequality gives
		\begin{align}
		|\langle T(t)x,y \rangle| &\leq 
		\frac{e^{\xi  t}}{2\pi t} 
		\int^{\infty}_{-\infty}
		\left| \langle R(\xi  + i\eta,A) x,R(\xi  - i\eta,A^*) y \rangle\right|
		d\eta  \notag \\
		&\leq 
		\frac{e^{\xi  t}}{2\pi t} 
		\left(
		\int_{-\infty}^{\infty} \|R(\xi  + i\eta,A) x\|^2 d\eta
		\right)^{\frac{1}{2}} \left(
		\int_{-\infty}^{\infty} \|R(\xi  + i\eta,A^*) y\|^2 d\eta
		\right)^{\frac{1}{2}}  \label{eq:Txy_bound}
		\end{align}
		for all $x,y \in H$, $\xi >0$, and $t >0$.
		Since $(T(t))_{t\geq 0}$ is uniformly bounded,
		it follows from the Plancherel theorem that 
		\begin{align*}
		\sup_{\xi  >0}~\! \xi  \int^{\infty}_{-\infty} \|R(\xi +i\eta,A^*)y\|^2 d\eta
		&=2\pi
		\sup_{\xi  >0} ~\!\xi 
		\int^{\infty}_{0} e^{-2\xi  t} \|T(t)^*y\|^2 d\eta \\
		&\leq 
		\pi M^2 \|y\|^2
		\end{align*}
		for every $y \in H$,
		where $M := \sup_{t\geq 0} \|T(t)\| = \sup_{t\geq 0} \|T(t)^*\|$.
		This and \eqref{eq:Txy_bound} establish
		\begin{equation}
		\label{eq:Tx_bound_trajectory}
		\|T(t)x\| \leq \frac{Me^{\xi  t} }{2t\sqrt{\pi \xi}}
		\left(
		\int_{-\infty}^{\infty} \|R(\xi  + i\eta,A) x\|^2 d\eta
		\right)^{\frac{1}{2}}
		\end{equation}
		for all $x \in H$, $\xi >0$, and $t >0$.
		
		Suppose that $x \in H$ satisfies \eqref{eq:gamma_LT_half} 
		for some $0 < \gamma < 1/2$ or 
		\eqref{eq:gamma_E_half}. In the latter case, we set $\gamma = 1/2$.
		For $\xi >0$, define
		\begin{equation}
		\label{eq:vartheta_def}
		h_{\gamma}(\xi) :=
		\begin{cases}
		\xi^{1-2\gamma} & \text{if $0<\gamma < 1/2$} \\
		\frac{1}{\log(1/\xi)} & \text{if $\gamma = 1/2$}.
		\end{cases}
		\end{equation}
		Taking $\xi = 1/t$,
		we have from \eqref{eq:Tx_bound_trajectory}
		that 
		for  every $\varepsilon>0$,
		there exists $t_0 >1$ such that for all $t \geq t_0$,
		\[
		\|T(t)(-A)^{-\alpha \gamma }x\| \leq 
		\frac{Me\varepsilon}{2\sqrt{\pi}} 
		\sqrt{
			\frac{1}{t h_\gamma(1/t)}}.
		\]
		By the definition \eqref{eq:vartheta_def} of 
		$h_\gamma$, we obtain \eqref{eq:gamma_LT_half_converse} for $0<\gamma<1/2$
		and 
		\eqref{eq:gamma_E_half_converse} for $\gamma = 1/2$.
		\hfill $\Box$	
	\end{proof}
	
	\begin{proof}[of Proposition~\ref{prop:decay_to_integral}]
		We easily see that 
		the assertions \ref{enu:polystable1} and \ref{enu:polystable2} hold,
		by combining
		Theorem~\ref{thm:decay_charac} with
		Lemmas~\ref{lem:frac_normalize1}, 
		\ref{lem:frac_normalize2}, and \ref{lem:decay_to_integral_trajectory}.
		
		It remains to prove that \ref{enu:polystable3} holds.
		By \ref{enu:polystable_trajectory2} of 
		Lemma~\ref{lem:decay_to_integral_trajectory},
		\[
		\sup_{t >2} 
		\sqrt{
			\frac{t}{\log t}}
		\big\|
		T(t)(-A)^{-\alpha/2} x
		\big\| < \infty
		\]
		for every $x \in H$.
		By the uniform boundedness principle, there exists $C_1>0$ such that
		\[
		\big\|
		T(t)(-A)^{-\alpha/2} 
		\big\| \leq  C_1\sqrt{
			\frac{\log t}{ t}}\qquad \forall t >2.
		\]
		Hence
		\begin{align*}
		\|T(t)(-A)^{-\alpha}\| &\leq 
		\big\|T(t/2)(-A)^{-\alpha/2} \big\|^2 \\
		&\leq 
		C_1^2 \frac{\log (t/2)}{t/2}
		\qquad \forall t >4.
		\end{align*}
		Thus \eqref{eq:TA_o} holds.
		\hfill $\Box$
	\end{proof}
	
	Let  $\xi  >0$ and $A$ be the generator of
	a uniformly bounded $C_0$-semigroup $(T(t))_{t\geq 0}$ on a Hilbert space $H$.
	Consider
	the Lyapunov equation
	\begin{equation}
	\label{eq:Lyap}
	\langle
	(A-\xi  I) x_1, Q(\xi )x_2 
	\rangle +
	\langle
	Q(\xi )x_1, (A-\xi  I) x_2
	\rangle = 
	-\langle
	x_1,x_2
	\rangle
	\end{equation}
	for $x_1,x_2 \in D(A)$. It
	has a unique self-adjoint solution $Q(\xi ) \in \mathcal{L}(H)$ given by
	\begin{equation}
	\label{eq:Q_def}
	Q(\xi )x
	:=
	\int^{\infty}_0 e^{-2\xi  t} T(t)^*T(t)x dt,\quad x \in H;
	\end{equation}
	see, e.g., Theorem~4.1.23 of \cite{Curtain1995}.
	
	We restate
	Proposition~\ref{prop:decay_to_integral} by using
	the self-adjoint solution of the Lyapunov equation \eqref{eq:Lyap}.
	\begin{lemma}
		\label{lem:Lyap}
		Let $(T(t))_{t\geq 0}$ be a uniformly bounded $C_0$-semigroup
		on a Hilbert space $H$ with generator 
		$A$ such that $i \mathbb{R} \subset \varrho(A)$, and let
		$Q(\xi )\in \mathcal{L}(H)$ be the self-adjoint solution of 
		the Lyapunov equation \eqref{eq:Lyap} for $\xi >0$.
		The following three assertions hold for a fixed $\alpha >0$:
		\begin{enumerate}
			\item  \label{enu:poly_Lyap1} 
			$\|T(t)(-A)^{-\alpha}\| = O(1/t)$ as $t \to \infty$ if and only if	 
			$Q(\xi )$
			satisfies
			\begin{align}
			\label{eq:Q_cond1}
			\lim_{\xi  \to 0+} \xi^{1-2\gamma}\langle(-A)^{-\alpha \gamma} 
			x, Q(\xi ) (-A)^{-\alpha \gamma}x\rangle  = 0
			\end{align}
			for all $x \in H$ and
			some/all $\gamma \in (0,1/2)$.
			\item  \label{enu:poly_Lyap2} 
			If $\|T(t)(-A)^{-\alpha}\| = O(1/t)$ as $t \to \infty$, then
			$Q(\xi )$ satisfies
			\begin{align}
			\label{eq:Q_cond2}
			\lim_{\xi  \to 0+} \frac{\langle(-A)^{-\alpha/2} 
				x, Q(\xi ) (-A)^{-\alpha/2}x\rangle}{\log(1/\xi)}  = 0
			\end{align}
			for all $x \in H$.
			\item \label{enu:poly_Lyap3} 
			If $Q(\xi )$ satisfies \eqref{eq:Q_cond2}  for all $x \in H$,
			then
			\begin{equation}
			\label{eq:TA_o2}
			\|T(t)(-A)^{-\alpha}\|  = O \left(
			\frac{\log t}{t} 
			\right)\qquad t\to \infty.
			\end{equation}
		\end{enumerate}
	\end{lemma}
	\begin{proof}
		Let $\xi  >0$ and $x \in H$.
		The Plancherel theorem shows that 
		\[
		\frac{1}{2\pi} 
		\int^{\infty}_{-\infty}
		\|R(\xi  + i\eta,A)x \|^2 d\eta = 
		\int^{\infty}_0 e^{-2\xi  t} \|T(t)x\|^2 dt.
		\]
		By the definition \eqref{eq:Q_def} of $Q(\xi )$, we obtain
		\[
		\langle x, Q(\xi ) x \rangle = 
		\frac{1}{2\pi} 
		\int^{\infty}_{-\infty}
		\|R(\xi  + i\eta,A)x \|^2 d\eta.
		\]
		Hence the assertions \ref{enu:poly_Lyap1}--\ref{enu:poly_Lyap3}
		immediately follow
		from Proposition~\ref{prop:decay_to_integral}.
		\hfill $\Box$	
	\end{proof}
	
	\section{Decay rate of Cayley transform}
	\label{sec:Decay_CT}
	In this section, we study the decay rate of
	the Cayley transform of the generator 
	of a polynomially stable $C_0$-semigroup.
	We start by establishing the discrete version of 
	the  normalization results on
	polynomial decay rates developed in 
	Lemmas~\ref{lem:frac_normalize1}--\ref{lem:frac_normalize3}.
	\begin{lemma}
		\label{lem:alpha_gamma}
		Let $A$ be the generator of
		a uniformly bounded $C_0$-semigroup on a Banach space $X$ 
		such that $0 \in \varrho(A)$.
		Suppose that $B \in \mathcal{L}(X)$ is power bounded and
		commutes with $A^{-1}$.
		Then the following assertions hold for
		constants $\alpha,\beta>0$ and a nondecreasing function
		$f:\mathbb{N} \to (0,\infty)$.
		\begin{enumerate}
			\item \label{enu:poly_dist1}
			The statements 
			{\rm (i)-(ii)} below are equivalent:\vspace{5pt}
			\begin{enumerate}
				\item $\|B^n (-A)^{-\alpha}\| = O \left(
				\dfrac{f(n)}{n}\right)$ as $n \to \infty$. \vspace{5pt}
				\item $\|B^n (-A)^{-\alpha\gamma }\| = O \left(
				\dfrac{f(n)^{\gamma}}{n^{\gamma}}
				\right)$ as $n \to \infty$ for some/all  $\gamma >0$.\vspace{5pt}
			\end{enumerate}
			\item \label{enu:poly_dist2}
			The statements 
			{\rm (i)--(ii)} below are equivalent:\vspace{5pt}
			\begin{enumerate}
				\item $\|B^n (-A)^{-\alpha} x\| = o \left(
				\dfrac{f(n)}{n} \right)$ as $n \to \infty$ for all $x \in X$.\vspace{5pt}
				\item  $\|B^n (-A)^{-\alpha\gamma }x\| = o \left(
				\dfrac{f(n)^{\gamma}}{n^{\gamma}}
				\right)$ as $n \to \infty$ for all $x \in X$  and some/all
				$\gamma >0$. \vspace{5pt}
			\end{enumerate}
			\item \label{enu:poly_dist3}
			If $x \in X$ satisfies \[
			\|B^n (-A)^{-\alpha} x\| = o \left(
			\dfrac{f(n)^\beta}{n^\beta} \right)\qquad n \to \infty,
			\]
			then
			\[
			\|B^n (-A)^{-\alpha \gamma } x\| = o \left(
			\dfrac{f(n)^{\beta\gamma}}{n^{\beta\gamma}} \right)\qquad n \to \infty
			\]
			for every $0<\gamma < 1$.
		\end{enumerate}
	\end{lemma}
	\begin{proof}
		\ref{enu:poly_dist1}	
		(ii) $\Rightarrow$ (i): 
		Let $\gamma>0$ and 
 		set $\delta := \alpha \gamma$. By assumption,
		there exists a constant $C >0$ such that 
		\begin{equation}
		\label{eq:BnA_bound}
		\|B^n (-A)^{-\delta}\| \leq C\left(
		\frac{f(n)}{n}\right)^{\gamma}\qquad \forall n \in \mathbb{N}.
		\end{equation}
		For every $a >0$,
		$B$ commutes with $(-A)^{-a}$
		by Proposition~3.1.1.f) of \cite{Haase2006}, and hence
		for every $x \in D((-A)^a)$, $Bx \in D((-A)^a)$ and
		$B(-A)^{a}x = (-A)^a Bx$ by Proposition~B.7 of \cite{Arendt2001}.
		Using \eqref{eq:BnA_bound}, we  obtain
		\begin{equation}
		\label{eq:k_powered_bound}
		\|B^{kn} (-A)^{-k\delta}\| \leq 
		\|B^n (-A)^{-\delta}\|^k \leq C^k\left(
		\frac{f(n)}{n}
		\right)^{k\gamma}\qquad \forall k,n \in \mathbb{N}.
		\end{equation}
		From the moment inequality (see, e.g., Proposition~6.6.4 of
		\cite{Haase2006}), it follows that for every $k \in \mathbb{N}$ and
		every 
		$\vartheta \in (0,1)$, there exists a constant $L_0 >0$ such that 
		\begin{align}
		\|B^{kn} (-A)^{-\vartheta k \delta }x\| &=
		\|(-A)^{k\delta (1-\vartheta)} B^{kn} (-A)^{-k \delta }x\| \notag \\
		&\leq 
		L_0\|(-A)^{k\delta } B^{kn} (-A)^{-k \delta }x\|^{1-\vartheta}
		~\!
		\|B^{kn} (-A)^{-k \delta }x\|^{\vartheta} \notag \\
		&\leq 
		L_0\|B^{kn}\|^{1-\vartheta}
		~\!\|B^{kn} (-A)^{-k \delta }\|^{\vartheta}~\! \|x\|
		\label{eq:B_moment_ineq}
		\end{align} 
		for every $x \in X$ and every $n \in \mathbb{N}$.
		Hence \eqref{eq:k_powered_bound} yields
		\begin{align}
		\label{eq:Adkn_bound}
		\|B^{kn} (-A)^{-\vartheta k \delta }\|
		\leq L_0 M^{1-\vartheta} C^{k\vartheta} \left(
		\frac{f(n)}{n}
		\right)^{k\gamma \vartheta}\qquad \forall n \in \mathbb{N},
		\end{align}
		where $M:= \sup_{n \in \mathbb{N}\cup \{0\}} \|B^n\|$.
		Choose a constant $k \in \mathbb{N}$ 
		satisfying $k > 1/\gamma$ and set 
		$\vartheta = 1/(k\gamma)$.
		Then 
		\[
		\|B^{kn} (-A)^{-\alpha }\| \leq 
		k L_1 M C^{1/\gamma} ~\!
		\frac{f(kn)}{kn}
		\qquad \forall n \in \mathbb{N}
		\]
		for some $L_1 >0$.
		Since 
		\[
		\frac{kn+\ell}{kn} \leq \frac{k(n+1)}{kn} \leq 2
		\]
		for every $\ell=0,\dots,k-1$ and every $n\in \mathbb{N}$,
		it follows that 
		\begin{align*}
		\|B^{kn + \ell} (-A)^{-\alpha }\| &\leq 
		\|B^{\ell}\| ~\! 
		\|B^{kn } (-A)^{-\alpha }\|  \\
		&\leq 
		2k L_1 M^{2} C^{1/\gamma} ~\!
		\frac{f (kn+\ell )}{kn+\ell}
		\end{align*}
		for every $\ell=0,\dots,k-1$ and every $n\in \mathbb{N}$.
		Thus, (i) holds. 
		
		(i) $\Rightarrow$ (ii): 
		Take $\tilde \gamma>0$.
		Substituting $\gamma=1$ and $\vartheta = \tilde \gamma / k$ with 
		$k > \tilde \gamma$ into \eqref{eq:Adkn_bound}, we obtain
		\[
		\|B^{kn} (-A)^{-\alpha \tilde \gamma }\| \leq 
		L_2 M C^{\tilde \gamma} \left(
		\frac{f(n)}{n}
		\right)^{\tilde \gamma}\qquad \forall n \in \mathbb{N}
		\]
		for some $L_2 >0$.
		Hence we obtain (ii) with $\gamma = \tilde \gamma$
		by similar arguments as above.
		
		\ref{enu:poly_dist2} Suppose that (ii) holds, and  
		set $\delta := \alpha \gamma$.
		Since
		\[
		\sup_{n \in \mathbb{N}}
		\left(
		\frac{n}{f(n)}
		\right)^{\gamma} \|B^n (-A)^{-\delta} x\| < \infty,
		\]
		for all $x \in X$,
		it follows from the uniform boundedness principle that 
		there exists a constant $C_1 >0$ such that
		\[
		\|B (-A)^{-\delta} \| \leq C_1 \left(
		\frac{f(n)}{n}
		\right)^{\gamma} \qquad \forall n \in \mathbb{N}.
		\]
		Let $x \in X$. Then
		\begin{align*}
		\|B^{kn} (-A)^{-k\delta } x\| &\leq 
		\|B^{n} (-A)^{-\delta } \|^{k-1}~\!
		\|B^{n} (-A)^{-\delta } x\| \\
		&\leq 
		C_1^{k-1}  \left(
		\frac{f(n)}{n}
		\right)^{(k-1)\gamma}\|B^{n} (-A)^{-\delta} x\|
		\qquad \forall k,n \in \mathbb{N}.
		\end{align*}
		This implies
		\[
		\|B^{kn} (-A)^{-k\delta } x\| = 
		o \left(
		\frac{f(n)^{k\gamma}}{n^{k\gamma}}
		\right)\qquad n \to \infty.
		\]
		The rest of the proof is the same as the proof of a). We omit the details.
		
		\ref{enu:poly_dist3} This assertion directly follows from
		the application of the moment inequality as in \eqref{eq:B_moment_ineq}.
		\hfill $\Box$	
	\end{proof}
	
	Using the argument based on
	Lyapunov equations developed in Section~\ref{sec:Lyap},
	we estimate the decay rate of 
	the Cayley transform.
	We use the following preliminary result obtained in 
	the proof of \cite[Theorem~4.3]{Guo2006}.
	\begin{lemma}
		\label{lem:yx_bound}
		Let $A$ be the generator of
		a uniformly bounded $C_0$-semigroup on a Hilbert space $H$. 
		Suppose that 
		the Cayley transform $A_d := (I+A) (I-A)^{-1}$ is power bounded.
		Let $r \in (0,1)$ and 
		\begin{equation}
		\label{eq:sig_r}
		2 \xi = \frac{1-r^2}{1+r^2}.
		\end{equation}
		Then the self-adjoint solution $Q(\xi)\in \mathcal{L}(H)$ of the Lyapunov equation \eqref{eq:Lyap}
		satisfies
		\begin{align}
		\label{eq:y_tilx_bound}
		(n+1) 
		|\langle 
		y, r^n A_d^n (I-A)^{-1}x
		\rangle| \leq 
		M \|y\|
		\sqrt{\frac{\langle x , Q(\xi) x \rangle}{1-r^4}}
		\end{align}
		for every $x,y \in H$ and every $n \in \mathbb{N}$,
		where $M := \sup_{n \in \mathbb{N} \cup\{0\}} \|A_d^n\|$.
	\end{lemma}
	\begin{theorem}
		\label{thm:DR_CT}
		Let $(T(t))_{t\geq 0}$ be a uniformly bounded $C_0$-semigroup on a Hilbert space $H$ with generator 
		$A$ such that $i \mathbb{R} \subset \varrho(A)$.
		Suppose that 
		the Cayley transform $A_d := (I+A) (I-A)^{-1}$ is power bounded.
		If $\|T(t)(-A)^{-\alpha}\| = O(1/t)$ as $t\to \infty$
		for some $\alpha >0$,
		then $A_d$ satisfies
		\begin{align}
		\label{eq:Ad_bound1}
		\|A_d^n (-A)^{-\alpha-2} x\| = o\left(
		\frac{\log n}{n}\right)\qquad n\to \infty
		\end{align}
		for every $x \in H$ and
		\begin{align}
		\label{eq:Ad_bound2}
		\|A_d^n (-A)^{-\alpha-2} \| = O\left(
		\frac{\log n}{n}\right)\qquad n\to \infty.
		\end{align}
	\end{theorem}
	\begin{proof}
		Take $r \in (0,1)$. Let $\xi>0$ be given in \eqref{eq:sig_r}
		and $Q(\xi)\in \mathcal{L}(H)$ be the self-adjoint solution  of the Lyapunov equation \eqref{eq:Lyap}.
		By Lemma~\ref{lem:yx_bound}, 
		$Q(\xi)$  satisfies 
		\eqref{eq:y_tilx_bound}  for every $x,y \in H$ and every $n \in \mathbb{N}$.

		Let $\varepsilon >0$, $y \in H$, and $x \in D((-A)^{\alpha/2})$.
		Define $\tilde x := 
		(I-A)^{-1}x$.
		By Lemma~\ref{lem:Lyap}.\ref{enu:poly_Lyap2},
		there exists $N_0 \in \mathbb{N}$ such that 
		\[
		\frac{\langle x, Q(\xi) x\rangle }{\log(1/\xi)} \leq \varepsilon^2
		\] 
		for 
		\begin{equation}
		\label{eq:sigma_def}
		2 \xi = \frac{2n+1}{2n^2+2n+1}
		\end{equation}
		with $n \geq N_0$.
		For $\xi$ given in \eqref{eq:sigma_def},
		\[
		r = \frac{n}{n+1}
		\]
		satisfies \eqref{eq:sig_r} and 
		\[
		\frac{\log(1/\xi)}{1-r^4} \leq C_1^2 n\log n\qquad \forall n \in \mathbb{N}.
		\]
		for some constant $C_1 >0$.
		Therefore, \eqref{eq:y_tilx_bound} yields
		\begin{align*}
		|\langle 
		y, A_d^n \tilde x
		\rangle| \leq MC_1 \|y\| \varepsilon  \cdot \frac{\sqrt{n\log n}}{n+1} \left(1 + \frac{1}{n}  \right)^n
		\qquad \forall n \geq N_0.
		\end{align*}
		This implies that 
		\[
		\|A_d^n \tilde x\| = o \left(
		\sqrt{\frac{\log n}{n}}
		\right)\qquad n\to \infty.
		\]
		Since
		\[
		\{
		\tilde x = (I-A)^{-1}x: x \in D((-A)^{\alpha/2})
		\} = \ran \big((-A)^{-\alpha/2 - 1}\big),
		\]
		we obtain
		\begin{equation}
		\label{eq:alpha_half}
		\|A_d^n (-A)^{-\alpha/2 - 1} x \| = o \left(
		\sqrt{\frac{\log n}{n}}
		\right)\qquad n\to \infty
		\end{equation}
		for every $x \in H$.
		The first assertion \eqref{eq:Ad_bound1} follows from 
		Lemma~\ref{lem:alpha_gamma}.
		Applying the uniform boundedness principle to \eqref{eq:Ad_bound1},
		we obtain \eqref{eq:Ad_bound2}.
		\hfill $\Box$	
	\end{proof}
	
	\begin{remark}
		In the proof of Theorem~\ref{thm:DR_CT},
		we employ \ref{enu:poly_Lyap2} of Lemma~\ref{lem:Lyap}.
		One can apply \ref{enu:poly_Lyap1} of Lemma~\ref{lem:Lyap}
		in a similar way and consequently obtain
		\[
		\|A_d^n (-A)^{-\alpha-1/\gamma} \| = O\left(
		\frac{1}{n}\right)\qquad n\to \infty
		\]
		for every $0<\gamma<1/2$.
		However, this result is less sharp than \eqref{eq:Ad_bound2}.
	\end{remark}

	We do not know whether  the
	logarithmic factor in Theorem~\ref{thm:DR_CT} may be dropped in general.
	In some cases, however, we can omit
	the logarithmic correction,
	as in Propositions~4.1 and 4.2 in \cite{Batkai2006}.
	It is worth mentioning that we do not need the assumption on
	the power boundedness of the Cayley transform.
	\begin{proposition}
		\label{prop:normal_case}
		Consider the following two cases:
		\begin{enumerate}
			\item \label{enu:normal}
			Let $(T(t))_{t\geq 0}$ be a uniformly bounded $C_0$-semigroup  on a Hilbert space  with generator 
			$A$ such that $i \mathbb{R} \subset \varrho(A)$ and $A$ is normal.
			\item  \label{enu:MO}
			Let $\Omega$ be a locally compact Hausdorff space and $\mu$
			be a $\sigma$-finite regular Borel measure on $\Omega$. Assume that 
			either
			\begin{enumerate}
				\item $X := L^p(\Omega,\mu)$ for $1\leq p < \infty$ and 
				$\phi :\Omega \to \mathbb{C}$ is measureable with essential range in
				the open left half-plane $\mathbb{C}_-$; or that
				\item $X := C_0(\Omega)$ and $\phi:\Omega \to \mathbb{C}$
				is continuous with $\overline{\phi(\Omega)} \subset \mathbb{C}_-$.
			\end{enumerate}
			Let $A$ be the multiplication operator induced by $\phi$ on $X$, i.e,
			$Af =\phi f$ with domain $D(A) := \{f \in X:\phi f \in X\}$ and $(T(t))_{t\geq 0}$
			be the $C_0$-semigroup on $X$ generated by $A$.
		\end{enumerate}
		In both cases \ref{enu:normal} and \ref{enu:MO},
		if  $\|T(t)(-A)^{-\alpha} \| = O(1/t)$ as $t\to \infty$ for
		some $\alpha>0$,
		then the Cayley transform $A_d := (I+A)(I-A)^{-1}$ satisfies
		\begin{align}
		\label{eq:Ad_bound_normal}
		\|A_d^n (-A)^{-\alpha-2} \| = O\left(
		\frac{1}{n}\right)\qquad n\to \infty.
		\end{align}
	\end{proposition}

	\begin{proof}
		The proof is divided up into two steps.
		In the first step, we characterize the norm of $A_d^n (-A)^{-\alpha}$ 
		for $n \in \mathbb{N}$ by the spectrum of $A$.
		In the second step, we obtain 
		the decay estimate \eqref{eq:Ad_bound_normal}
		from this characterization and the geometrical condition on
		$\sigma(-A)$ for polynomial decay given in Propositions~4.1 and 4.2 of
		\cite{Batkai2006}.
		
		{\em Step~1:}
		First, we consider the case \ref{enu:normal}.
		Define 
		\begin{equation}
		\label{eq:f_n_def}
		f_n(\lambda) := 
		\left( 
		\frac{1-\lambda}{1+\lambda}
		\right)^{n}\lambda^{-\alpha-2}
		\end{equation} for  $\lambda \in \mathbb{C} \setminus (-\infty,0]$ 
		and $n \in\mathbb{N}$. Then $A_d^n (-A)^{-\alpha-2} = f_n(-A)$
		by the product formula of functional calculi (see, e.g.,	
		Theorem~1.3.2.c) of \cite{Haase2006}).
		The spectral mapping theorem 
		(see, e.g., Theorem~2.7.8 of \cite{Haase2006})
		shows that 
		\[
		\sigma(f_n(-A)) = f_n(\tilde \sigma(-A)),
		\]
		where $f_n(\infty) := \lim_{\lambda \to 0}f_n(1/\lambda) = 0$.
		Since $A$ is normal, we see that 
		$f_n(-A)$ is also normal, by using 
		a multiplication operator unitarily equivalent to $A$ 
		obtained from the spectral 
		theorem; see, for example, Theorem~4.1 of \cite{Haase2018} for 
		the multiplicator version of
		the spectral theorem for unbounded normal operators.
		Moreover, $f_n(-A)$ is bounded. Hence
		the spectral radius of $f_n(-A)$ equals $\|f_n(-A)\|$; see, e.g., 
		Theorem~5.44 of \cite{Weidmann1980}.
		This yields
		\begin{equation}
		\label{eq:SMT}
		\|A_d^n (-A)^{-\alpha-2}\| = 
		\|f_n(-A)\| = \sup_{\lambda \in \sigma(f_n(-A))} |\lambda| =
		\sup_{\lambda \in \tilde \sigma(-A)} |f_n(\lambda)|.
		\end{equation}
		
		A result similar to \eqref{eq:SMT} is obtained in the case 
		\ref{enu:MO}. 
		Let $h_{\text{ess}}(\Omega)$ denote the essential range of 
		a measurable function $h:\Omega \to \mathbb{C}$.
		Define the function $f_n$ as in \eqref{eq:f_n_def}.
		In the $L^p$-case (i),
		\begin{align*}
		\|A_d^n(-A)^{-\alpha-2}\| &= 
		\sup \big\{
		|z| : z \in \big(f_n \circ (-\phi)\big)_{\text{ess}}(\Omega)
		\big\}  \\
		\sigma(A) &= \phi_{\text{ess}}(\Omega)
		\end{align*}
		by Proposition~I.4.10 of \cite{Engel2000}.
		Moreover, one can obtain
		\[
		\sup\big\{|z| : z \in \big(f_n \circ (-\phi)\big)_{\text{ess}}(\Omega) \big\} = 
		\sup \{
		|f_n(\lambda)| : -\lambda \in \phi_{\text{ess}}(\Omega)
		\}.
		\]
		Therefore,
		\begin{equation}
		\label{eq:LpC0_norm}
		\|A_d^n (-A)^{-\alpha-2}\| = \sup_{\lambda \in \sigma(-A)}|f_n(\lambda)|.
		\end{equation}
		In the $C_0$-case (ii), we also obtain \eqref{eq:LpC0_norm}, since 
		\begin{align*}
		\|A_d^n(-A)^{-\alpha-2}\| &=
		\sup_{s \in \Omega} \big|\big(f_n \circ (-\phi)\big)(s)\big| \\
		\sigma(A) &= \overline{\phi(\Omega)}
		\end{align*}
		by Proposition~I.4.2 of \cite{Engel2000}.

		{\em Step~2:}
		Propositions~4.1 and 4.2 of \cite{Batkai2006} show that 
		in both cases \ref{enu:normal} and \ref{enu:MO},
		$\|T(t)(-A)^{-\alpha}\| = O(1/t)$ as $t\to \infty$ if and only if
		there exist $\delta,C >0$ such that 
		$|\im \lambda| \geq C (\re \lambda)^{-1/\alpha}$ for every $\lambda \in 
		\sigma(-A)$ with $\re \lambda \leq \delta$.
		Define
		\begin{align*}
		\Omega_1 := 
		\{
		\lambda \in \sigma(-A):
		\re \lambda \leq \delta
		\},\quad 
		\Omega_2 := \sigma(-A) \setminus \Omega_1.
		\end{align*}
		In what follows,
		we estimate $|n f_n(\lambda)|$ for $\lambda \in \sigma(-A)$, by
		dividing the argument into two cases: $\lambda \in \Omega_1$
		and $\lambda \in \Omega_2$.
		To this end, we use the identities
		\begin{equation}
		\label{eq:1pmlambda}
		\left|
		\frac{1-\lambda}{1+\lambda}
		\right|^n = 
		\left(
		\frac{(1-\re \lambda)^2 + |\im \lambda|^2}
		{(1+\re \lambda)^2 + |\im \lambda|^2}
		\right)^{\frac{n}{2}} = 
		\left(
		1 - \frac{4 \re \lambda}{|1+\lambda|^2}
		\right)^{\frac{n}{2}}
		\end{equation}
		for $\lambda \in \sigma(-A)$.
		
		First, we consider the case $\lambda \in \Omega_1 \subset
		\{
		s \in \mathbb{C}: |\im \lambda| \geq C (\re \lambda)^{-1/\alpha}
		\}$.
		Then
		\[
		1 \geq  
		\frac{4 \re \lambda}{|1+\lambda|^2} \geq 
		\frac{
			\frac{4C^{\alpha}}{|\im \lambda|^{\alpha}}
		}{|1+\lambda|^2} \geq
		\frac{4C^{\alpha}}{C_1|\im \lambda|^{\alpha+2}} > 0
		\]
		for some $C_1 \geq 1$. Set $C_2 := 4C^{\alpha}/C_1$.
		Then
		\[
		n|f_n(\lambda)| \leq 
		\frac{n}{|\im \lambda|^{\alpha+2}}
		\left(
		1 -	\frac{C_2}{|\im \lambda|^{\alpha+2}}
		\right)^{\frac{n}{2}} 
		\]
		for all $n \in \mathbb{N}$.
		Define 
		\[
		g_n(s) := \frac{n}{s} \left( 
		1 - \frac{C_2}{s}
		\right)^{\frac{n}{2}}
		\]
		for $s \geq C_2$ and $n \in \mathbb{N}$.
		Then
		\[
		g_n'(s) =
		\frac{
			n\left(
			1 - \frac{C_2}{s}
			\right)^{\frac{n}{2}-1} (C_2(n+2) - 2s)
		}{2s^3}\qquad \forall s > C_2.
		\]
		Hence $g_n$ takes the maximum value at $s = s^* :=\frac{C_2(n+2)}{2}$
		on $[C_2,\infty)$ and 
		\[
		g_n(s^*) = \frac{2n}{C_2(n+2)}
		\left(
		1 - \frac{2}{n+2}
		\right)^{\frac{n}{2}} \to \frac{2}{eC_2}\qquad n \to \infty.
		\]
		This means that 
		\begin{equation}
		\label{eq:bound1_normal}
		\sup_{n \in \mathbb{N}} \sup_{\lambda \in \Omega_1}n|f_n (\lambda)| < \infty.
		\end{equation}

		Next we investigate the case $\lambda \in \Omega_2$.
		We have that 
		\begin{align*}
		1 \geq 
		\frac{4 \re \lambda}{|1+\lambda|^2} \geq
		\frac{4\delta}{(1+|\lambda|)^2} \geq 
		\frac{4\delta}{(1+1/\delta)^2 |\lambda|^2} >0.
		\end{align*}
		Set $C_3 := 4\delta / (1+1/\delta)^2$.
		Then
		\[
		n|f_n(\lambda)| \leq 
		\frac{1}{\delta^{\alpha}}
		\frac{n}{|\lambda|^{2}}
		\left(
		1 -	\frac{C_3}{|\lambda|^{2}}
		\right)^{\frac{n}{2}}  .
		\]
		Similarly to the case $\lambda \in \Omega_1$, we obtain
		\begin{equation}
		\label{eq:bound2_normal}
		\sup_{n \in \mathbb{N}}\sup_{\lambda \in \Omega_2} n|f_n (\lambda)| < \infty.
		\end{equation}
		Combining  \eqref{eq:bound1_normal} and \eqref{eq:bound2_normal},
		we have that 
		\[
		\sup_{n \in \mathbb{N}} n\|A_d^n (-A)^{-\alpha-2}\| < \infty.
		\]
		Thus, the desired conclusion \eqref{eq:Ad_bound_normal} holds.
		\hfill $\Box$	
	\end{proof}
	
	\begin{example}
		There exists a normal operator $A$ such that 
		the norm estimate of Cayley transforms 
		in Proposition~\ref{prop:normal_case} 
		cannot be improved.
		Consider a diagonal operator $A: D(A) \subset \ell^2(\mathbb{C})
		\to \ell^2(\mathbb{C})$  defined by
		\[
		Ax := \sum_{k \in \mathbb{N}} \left(
		-\frac{1}{k} + ik
		\right) \langle x,e_k \rangle e_k
		\]
		with domain
		\[
		D(A) := 
		\left\{
		x \in \ell^2(\mathbb{C}) : \sum_{k \in \mathbb{N}} 
		k^2 |\langle x,e_k \rangle |^2 < \infty
		\right\},
		\]
		where  $(e_k)_{k \in \mathbb{N}}$ is the 
		standard basis of $ \ell^2(\mathbb{C})$.
		The semigroup $(T(t))_{t\geq 0}$ generated by $A$
		satisfies $\|T(t)A^{-1}\| = O(1/t)$ as $t\to \infty$ 
		by Proposition~4.1 of \cite{Batkai2006}.
		We see from Proposition~\ref{prop:normal_case}  that 
		\begin{equation}
		\label{eq:Ad_decay_ex1}
		\|A_d^n A^{-3}\| = O\left(\frac{1}{n} \right)\qquad n \to \infty.
		\end{equation}
		The Cayley transform $A_d$ is given by
		\[
		A_dx = \sum_{k \in \mathbb{N}} \frac{k-1+ik^2}{k+1-ik^2}
		\langle x ,e_k \rangle e_k\qquad \forall x \in X.
		\]
		Since
		\[
		\left|
		\frac{k-1+ik^2}{k+1-ik^2}
		\right|^2 = \frac{(k-1)^2+k^4}{(k+1)^2+k^4}\leq 1
		\]
		for all $k \in \mathbb{N}$,
		it follows that $A_d$ is power bounded.
		By
		Lemma~\ref{lem:alpha_gamma}.\ref{enu:poly_dist1},
		the norm estimate \eqref{eq:Ad_decay_ex1} is equivalent to
		\begin{equation}
		\label{eq:Ad_decay_ex2}
		\|A_d^n A^{-1}\| = O\left(\frac{1}{n^{1/3}} \right)\qquad n \to \infty.
		\end{equation}

		Define $\lambda_k := 1/k - ik \in \sigma(-A)$ for $k \in \mathbb{N}$,
		and take $m \in \mathbb{N}$.
		Then 
		\[
		\|A_d^{mn^3} A^{-3} \| =
		\sup_{k \in \mathbb{N}}
		\left|
		\frac{1-\lambda_k}{1 + \lambda_k}
		\right|^{mn^3} |\lambda_k|^{-3}
		\geq 
		\left|
		\frac{1-\lambda_n}{1 + \lambda_n}
		\right|^{mn^3} |\lambda_n|^{-3}.
		\]
		for every $n \in \mathbb{N}$. 
		We have  that 
		\begin{align*}
		n^3 \left|
		\frac{1-\lambda_n}{1 + \lambda_n}
		\right|^{mn^3} |\lambda_n|^{-3} &=
		\frac{n^3}{\left(
			\frac{1}{n^2} + n^2
			\right)^{\frac{3}{2}}}
		\left(
		1- \frac{4}{n^3+n+2+1/n}
		\right)^{\frac{mn^3}{2}} .
		\end{align*} 
		Since
		\[
		\left(
		1- \frac{4}{n^3+n+2+1/n}
		\right)^{\frac{mn^3}{2}} \to e^{-2m} \qquad n \to \infty,
		\]
		it follows that 
		\begin{equation}
		\label{eq:Ad_decay_liminf}
		\liminf_{n \to \infty} n^3  \|A_d^{mn^3} A^{-3} \| \geq e^{-2m}.
		\end{equation}
		By
		\eqref{eq:Ad_decay_liminf} with $m=1$,
		the norm estimate \eqref{eq:Ad_decay_ex1} 
		is optimal
		in the sense that 
		$\limsup_{n \to \infty} n \|A_d^n A^{-3}\| >0$.
		The norm estimate
		\eqref{eq:Ad_decay_ex2} and 
		the substitution of $m=3$ into
		\eqref{eq:Ad_decay_liminf} imply that 
		the optimal decay rate of $\|A_d^n A^{-1}\|$ is $1/n^{1/3}$.
	\end{example}
	
	\section{Robustness analysis of polynomial stability}
	\label{sec:robustness}
	As another application of the argument
	based on Lyapunov equations established in Section~\ref{sec:Lyap},
	we here extend to the case of polynomial stability
	the following result (Lemma~2.6 of \cite{Guo2006}) on the preservation
	of uniform boundedness.
	\begin{lemma}
		\label{lem:perturb_bounded}
		Let 
		$A$ be the generator of a uniformly bounded $C_0$-semigroup  on a Hilbert space $H$. 
		If $0 \in \varrho(A)$, then
		$A+rA^{-1}$ also generates a uniformly bounded $C_0$-semigroup
		semigroup on $H$ for every $r \geq 0$.
	\end{lemma}
	
	This robustness result can be extended to the case of 
	strong stability and exponential stability; see p.~358 of \cite{Guo2006}.
	
	It turns out that 
	the class of perturbations $\{rA^{-1}:r>0\}$
	results in at most only arbitrarily small loss  of decay rates.
	\begin{proposition}
		\label{prop:perturbation}
		Let 
		$A$ be the generator of a 
		uniformly bounded $C_0$-semigroup
		$(T(t))_{t\geq 0}$ on a Hilbert space $H$ 
		such that $i \mathbb{R} \subset \varrho(A)$.
		If
		$\|T(t)(-A)^{-\alpha}\| = O(1/t)$ as $t\to \infty$ for some $\alpha >0$, then 
		the following assertions hold for every $r \geq 0$:
		\begin{enumerate}
			\item $A+rA^{-1}$ generates a uniformly bounded $C_0$-semigroup 
			$(S_r(t))_{t\geq 0}$ on $H$;
			\item $i \mathbb{R} \subset \varrho(A+r A^{-1})$; and
			\item 
			for every $\varepsilon \in (0,1)$,
			\begin{equation}
			\label{eq:Sr_conv}
			\|S_r(t) (-A-rA^{-1})^{-\alpha}\| = 
			\begin{cases}
			O (1/t) & \text{if $\alpha>2$} \\
			O(\log t/ t) & \text{if $\alpha = 2$} \\
			O(1/t^{1-\varepsilon}) & \text{if $0<\alpha<2$}
			\end{cases}
			\end{equation}
			as $t\to \infty$.
		\end{enumerate}
	\end{proposition}
	
	We need three auxiliary results
	for the proof of Proposition~\ref{prop:perturbation}.
	First, we present a simple result on the resolvent set of $A+A^{-1}$.
	\begin{lemma}
		\label{lem:A_Ainv_resol}
		Let $A:D(A) \subset X \to X$ be a closed operator on 
		a Banach space $X$.
		If $i\mathbb{R} \subset \varrho(A)$, then
		$i\mathbb{R} \subset \varrho(A + A^{-1})$.
	\end{lemma}
	\begin{proof}
		Let $\omega \in \mathbb{R}$. A routine calculation shows that 
		\[
		i \omega - A-A^{-1} = 
		-(A- i\omega_1 I) (A- i\omega_2 I) A^{-1},
		\]
		where 
		\[
		\omega_1 := \frac{\omega}{2} + \sqrt{1+\frac{\omega^2}{4}},\qquad 
		\omega_2 := \frac{\omega}{2} - \sqrt{1+\frac{\omega^2}{4}}.
		\]
		From $i\mathbb{R} \subset \varrho(A)$, it follows that 
		$A- i\omega_1 I$ and $A- i\omega_2 I$ are invertible in $\mathcal{L}(X)$.
		Hence
		\[
		(i \omega - A-A^{-1})^{-1} 
		= -A (A- i\omega_1 I)^{-1}(A- i\omega_2 I)^{-1} \in \mathcal{L}(X).
		\]
		Thus, $i\mathbb{R} \subset \varrho(A + A^{-1})$.
		\hfill $\Box$	
	\end{proof}
	
	Second, we present a result on 
	the domain of fractional powers under bounded perturbations,
	which is used to prove \eqref{eq:Sr_conv} in the case $0<\alpha<2$.
	\begin{lemma}
		\label{lem:A_AD_domain}
		Let $X$ be a Banach space and $B \in \mathcal{L}(X)$. 
		Suppose that $A$ and $A+B$ generate 
		$C_0$-semigroups on $X$  with negative growth bounds. 
		Then the inclusion $D((-A)^\alpha) \subset D((-A-B)^{\beta})$
		holds for every $\alpha,\beta \in (0,1)$ with $\beta < \alpha$.
	\end{lemma}
	\begin{proof}
		For $\alpha \in (0,1)$,
		define the abstract H\"older spaces of order $\alpha$ by
		\begin{align*}
		X_{\alpha} &:= 
		\left\{
		x \in X: \lim_{\lambda \to \infty} \|\lambda^\alpha AR(\lambda,A) x\| = 0
		\right\} \\
		X_{\alpha}^B &:= 
		\left\{
		x \in X: \lim_{\lambda \to \infty} \|\lambda^\alpha (A+B)R(\lambda,A+B) x\| = 0
		\right\}.	
		\end{align*}
		Since $A$ and $A+B$ generate 
		$C_0$-semigroups  with negative growth bounds,
		it follows from Proposition~II.5.33 of \cite{Engel2000} that
		\begin{equation}
		\label{eq:AH_domain}
		X_{\alpha} \subset D\big((-A)^\beta\big) \subset X_{\beta},\quad 
		X_{\alpha}^B \subset D\big((-A-B)^\beta\big) \subset X_{\beta}^B
		\end{equation}
		for  
		$0 <\beta < \alpha< 1$.
		It suffices to show that 
		$X_{\alpha} \subset X_{\alpha}^B$ for every $\alpha \in (0,1)$.  In fact, 
		combining this inclusion 
		with \eqref{eq:AH_domain}, we obtain
		\[
		D\big((-A)^\alpha\big) \subset X_{\alpha} \subset X_{\alpha}^B \subset D\big((-A-B)^{\beta }\big)
		\]
		for $0 < \beta<\alpha<1$.
		
		To obtain $X_{\alpha} \subset X_{\alpha}^B$, we use 
		the identity
		\[
		R(\lambda,A+B) = R(\lambda,A) + R(\lambda,A)BR(\lambda,A+B)\qquad 
		\forall \lambda >0.
		\]
		This yields
		\begin{align}
		(A+B) R(\lambda,A+B) &= 
		AR(\lambda,A) + BR(\lambda,A) \notag \\ &\qquad 
		+ (A+B) R(\lambda,A) BR(\lambda,A+B)
		\label{eq:resol_decom}
		\end{align}
		for every $\lambda >0$.
		There exists $M>0$ such that 
		\[
		\|\lambda R(\lambda,A)\| \leq M,\quad 
		\|\lambda R(\lambda,A+B)\| \leq M
		\qquad \forall \lambda > 0,
		\]
		and hence
		\begin{align*}
		\|(A+B)R(\lambda,A)\| &\leq \|-I + \lambda R(\lambda,A)\|  + 
		\|BR(\lambda,A)\| \\
		&\leq 1+M + \frac{M\|B\|}{\lambda}\qquad \forall \lambda >0.
		\end{align*}
		Consequently,
		\begin{align*}
		&\lim_{\lambda \to \infty}
		\|\lambda^{\alpha }BR(\lambda,A) \| 
		= 0\\
		&\lim_{\lambda \to \infty}
		\|\lambda^{\alpha }(A+B) R(\lambda,A) BR(\lambda,A+B)\| 
		= 0
		\end{align*}
		for every $\alpha \in (0,1)$.
		It follows from \eqref{eq:resol_decom} that 
		$X_\alpha \subset X_{\alpha}^B$.
		\hfill $\Box$	
	\end{proof}
	
	Third,
	Lyapunov equations
	for $A$ are connected to those for $A+A^{-1}$ by
	the following result obtained in the proof of \cite[Lemma~2.6]{Guo2006}.
	\begin{lemma}
		\label{lem:Q1_Lyap}
		Let 
		$A$ be the generator of a uniformly bounded $C_0$-semigroup  on a Hilbert space $H$. 
		Suppose that $0 \in \varrho(A)$, and take $\kappa > \|A^{-1}\|^2$.
		For every $\xi >0$,
		the self-adjoint solution $Q(\xi)\in \mathcal{L}(H)$ of the Lyapunov equation \eqref{eq:Lyap}
		satisfies
		\[
		(A+A^{-1}-\xi I)^* Q\left( \frac {\xi}{1+\kappa} \right) + Q\left( \frac {\xi}{1+\kappa} \right) (A+A^{-1}-\xi I)  \leq 
		-I \quad \text{on $D(A)$}.
		\]
	\end{lemma}
	
	We are now ready to prove Proposition~\ref{prop:perturbation}.
	\begin{proof}[of Proposition~\ref{prop:perturbation}]
		Since
		\[
		A+rA^{-1} = r^{1/2} 
		\Big(
		\big( r^{-1/2}A \big) + 	\big( r^{-1/2}A \big) ^{-1}
		\Big)\qquad \forall r > 0,
		\]
		it suffices to prove the case $r = 1$.
		By Lemma~\ref{lem:perturb_bounded}, $(S_1(t))_{t\geq 0}$ is uniformly bounded.
		We see from Lemma~\ref{lem:A_Ainv_resol} that 
		$i\mathbb{R} \subset \varrho(A + A^{-1})$.
		
		It remains to prove the norm estimate 
		\eqref{eq:Sr_conv} for $r=1$.
		For $\xi >0$,
		let $Q(\xi)\in \mathcal{L}(H)$ be the self-adjoint solution of the Lyapunov equation \eqref{eq:Lyap}.
		Take $\kappa > \|A^{-1}\|^2$ and define
		\[
		Q_1(\xi) := Q\left( \frac {\xi}{1+\kappa} \right),\quad \xi >0.
		\]
		By Lemma~\ref{lem:Q1_Lyap},
		$Q_1(\xi)$ satisfies
		
		\[
		(A+A^{-1}-\xi I)^* Q_1(\xi) + Q_1(\xi) (A+A^{-1}-\xi I)  \leq 
		-I \quad \text{on $D(A)$}
		\]
		for every $\xi >0$.	
		Therefore,
		\begin{equation}
		\label{eq:resol_int_bound}
		\frac{1}{2\pi} 
		\int_{-\infty}^{\infty} \|R(\xi +i\eta, A+A^{-1}) x\|^2 d\eta \leq 
		\langle
		x, Q_1(\xi)x 
		\rangle
		\end{equation}
		for every $x \in H$ and every $\xi >0$; see Remark~2.3 of \cite{Guo2006}. 
		For $\gamma \in (0,1/2]$,
		define
		\[
		h_\gamma(\xi) := 
		\begin{cases}
		\xi^{1-2\gamma} & \text{if $0<\gamma < 1/2$} \\
		\frac{1}{\log(1/\xi)} & \text{if $\gamma = 1/2$}.
		\end{cases}
		\]
		Then \ref{enu:poly_Lyap1}  and \ref{enu:poly_Lyap2} of 
		Lemma~\ref{lem:Lyap}
		show that 
		\begin{align}
		\lim_{\xi \to 0+} h_{\gamma}(\xi)\langle x, Q_1(\xi) x\rangle 
		&= 0.\label{eq:Q1_bound}
		\end{align}
		for every $x \in D((-A)^{\alpha \gamma })$.
		
		First we consider the case $\alpha \geq 2$.
		Define $\gamma := 1/\alpha $. Then $\gamma \in (0,1/2]$
		and $\alpha \gamma = 1$. Since
		$D((-A)^{\alpha\gamma}) = D((-A-A^{-1})^{\alpha \gamma})$,
		it follows from \eqref{eq:resol_int_bound} and \eqref{eq:Q1_bound} that
		\begin{equation}
		\label{eq:resolvent_perturb}
		\lim_{\xi \to 0+} h_{\gamma}(\xi)
		\int_{-\infty}^{\infty} \|R(\xi +i\eta, A+A^{-1}) x\|^2 d\eta = 0
		\end{equation}
		for every $x \in D((-A-A^{-1})^{\alpha\gamma})$.
		Hence \eqref{eq:Sr_conv} holds in the case $\alpha \geq 2$ by
		Proposition~\ref{prop:decay_to_integral}.\ref{enu:polystable1} and
		\ref{enu:polystable3}.
		
		Next we investigate the case $0<\alpha < 2$.
		Choose $\gamma \in (0,1/2)$ and 
		$\delta \in (0,1-\alpha \gamma)$ arbitrarily.
		For every $\alpha_1,\xi_1>0$, 
		\begin{align*}
		D\big(
		(-A)^{\alpha_1}
		\big) &= D\big(
		(-A+\xi_1I)^{\alpha_1}
		\big) \\
		D\big(
		(-A-A^{-1})^{\alpha_1}
		\big) &= D\big(
		(-A-A^{-1} + \xi_1I)^{\alpha_1}
		\big)
		\end{align*}
		by Proposition~3.1.9.a) of \cite{Haase2006}.
		Since $A-\xi_1I$ and $A+A^{-1} - \xi_1I$ generate
		$C_0$-semigroups with negative growth bounds
		for every $\xi_1 >0$,
		it follows from
		Lemma~\ref{lem:A_AD_domain} that
		\begin{align*}
		D\big((-A-A^{-1})^{\alpha \gamma + \delta}\big) 
		\subset D\big((-A)^{\alpha \gamma }\big).
		\end{align*} 
		Therefore,  \eqref{eq:resol_int_bound} and \eqref{eq:Q1_bound} lead to
		\begin{equation*}
		\lim_{\xi \to 0+} \xi^{1-2\gamma}
		\int_{-\infty}^{\infty} \|R(\xi +i\eta, A+A^{-1}) x\|^2 d\eta = 0
		\end{equation*}
		for every $x \in D((-A-A^{-1})^{\alpha\gamma+ \delta})$.
		From Proposition~\ref{prop:decay_to_integral}.\ref{enu:polystable1},
		we have that
		\[
		\|
		T(t)(-A-A^{-1})^{-\alpha - \delta / \gamma}  
		\| = O\left(
		\frac{1}{t} 
		\right)\qquad t\to \infty.
		\]
		Since $\gamma \in (0,1/2)$ and 
		$\delta  \in (0,1-\alpha \gamma)$ were arbitrary,
		it follows from Lemma~\ref{lem:frac_normalize1} that
		\eqref{eq:Sr_conv} holds in the case $0<\alpha < 2$.
		\hfill $\Box$	
	\end{proof}


\begin{thebibliography}{10}
	\providecommand{\url}[1]{{#1}}
	\providecommand{\urlprefix}{URL }
	\expandafter\ifx\csname urlstyle\endcsname\relax
	\providecommand{\doi}[1]{DOI~\discretionary{}{}{}#1}\else
	\providecommand{\doi}{DOI~\discretionary{}{}{}\begingroup
		\urlstyle{rm}\Url}\fi
	
	\bibitem{Arendt2001}
	Arendt, W., Batty, C.J.K., Hieber, M., Neubrander, F.: Vector-valued Laplace
	Transforms and Cauchy Problems.
	\newblock Basel: Birkh\"auser (2001)
	
	\bibitem{Batkai2006}
	B\'atkai, A., Engel, K.J., Pr\"uss, J., Schnaubelt, R.: Polynomial stability of
	operator semigroups.
	\newblock Math. Nachr. \textbf{279}, 1425--1440 (2006)
	
	\bibitem{Batty2016}
	Batty, C.J.K., Chill, R., Tomilov, Y.: Fine scales of decay of operator
	semigroups.
	\newblock J. Eur. Math. Soc \textbf{18}, 853--929 (2016)
	
	\bibitem{Batty2008}
	Batty, C.J.K., Duyckaerts, T.: {Non-uniform stability for bounded semi-groups
		on Banach spaces}.
	\newblock J. Evol. Equ. \textbf{8}, 765--780 (2008)
	
	\bibitem{Batty2021}
	Batty, C.J.K., Gomilko, A., Tomilov, Y.: {A Besov algebra calculus for
		generators of operator semigroups and related norm-estimates}.
	\newblock Math. Ann. \textbf{379}, 23--93 (2021)
	
	\bibitem{Borichev2010}
	Borichev, A., Tomilov, Y.: Optimal polynomial decay of functions and operator
	semigroups.
	\newblock Math. Ann. \textbf{347}, 455--478 (2010)
	
	\bibitem{Cohen2016}
	Cohen, G., Lin, M.: {Remarks on rates of convergence of powers of
		contractions}.
	\newblock J. Math. Anal. Appl. \textbf{436}, 1196--1213 (2016)
	
	\bibitem{Curtain1995}
	Curtain, R.F., Zwart, H.J.: {An Introduction to Infinite-Dimensional Linear
		Systems Theory}.
	\newblock New York: Springer (1995)
	
	\bibitem{Eisner2010}
	Eisner, T.: Stability of Operators and Operator Semigroups.
	\newblock Basel: Birkh\"auser (2010)
	
	\bibitem{Engel2000}
	Engel, K.J., Nagel, R.: One-Parameter Semigroups for Linear Evolution
	Equations.
	\newblock New York: Springer (2000)
	
	\bibitem{Gautschi1959}
	Gautschi, W.: Some elementary inequalities relating to the gamma and incomplete
	gamma function.
	\newblock J. Math. Phys. \textbf{38}, 77--81 (1959)
	
	\bibitem{Gomilko2004}
	Gomilko, A.: {Cayley transform of the generator of a uniformly bounded
		$C_0$-semigroup of operators}.
	\newblock Ukrainian Math. J. \textbf{56}, 1212--1226 (2004)
	
	\bibitem{Gomilko2017}
	Gomilko, A.: Inverses of semigroup generators: a survey and remarks.
	\newblock In: \'Etudes op\'eratorielles, Banach Center Publ, pp. 107--142
	(2017)
	
	\bibitem{Gomilko2007}
	Gomilko, A., Zwart, H.: {The Cayley transform of the generator of a bounded
		$C_0$-semigroup}.
	\newblock Semigroup Forum \textbf{74}, 140--148 (2007)
	
	\bibitem{Gomilko2011}
	Gomilko, A., Zwart, H., Besseling, N.: Growth of semigroups in discrete and
	continuous time.
	\newblock Studia Math. \textbf{206}, 273--292 (2011)
	
	\bibitem{Gomilko1999}
	Gomilko, A.M.: {Conditions on the generator of a uniformly bounded
		$C_0$-semigroup}.
	\newblock Funct. Anal. Appl. \textbf{33}, 294--296 (1999)
	
	\bibitem{Guo2006}
	Guo, B.Z., Zwart, H.: {On the relation between stability of continuous-and
		discrete-time evolution equations via the Cayley transform}.
	\newblock Integr. equ. oper. theory \textbf{54}, 349--383 (2006)
	
	\bibitem{Haase2006}
	Haase, M.: The Functional Calculus for Sectorial Operators.
	\newblock Basel: Birkh\"auser (2006)
	
	\bibitem{Haase2018}
	Haase, M.: Lectures on Functional Calculus.
	\newblock 21st International Internet Seminar, Kiel Univ (2018).
	\newblock
	Retrieved May 27, 2021, from
	\url{https://www.math.uni-kiel.de/isem21/en/course/phase1/isem21-lectures-on-functional-calculus}
	
	\bibitem{Liu2005PDR}
	Liu, Z., Rao, B.: Characterization of polynomial decay rate for the solution of
	linear evolution equation.
	\newblock Angew. Math. Phys. \textbf{56}, 630--644 (2005)
	
	\bibitem{Nagy1970}
	Nagy, B.S., Foia\c{s}, C.: {Harmonic Analysis of Operators on Hilbelt Space}.
	\newblock North Holland Publishing Co. (1970)
	
	\bibitem{Ng2020}
	Ng, A.C.S., Seifert, D.: {Optimal rates of decay in the Katznelson-Tzafriri
		theorem for operators on Hilbert spaces}.
	\newblock J. Funct. Anal. \textbf{279}, Art. no. 108799 (2020)
	
	\bibitem{Paunonen2011}
	Paunonen, L.: {Perturbation of strongly and polynomially stable Riesz-spectral
		operators}.
	\newblock Systems Control Lett. \textbf{60}, 234--248 (2011)
	
	\bibitem{Paunonen2012SS}
	Paunonen, L.: Robustness of strongly and polynomially stable semigroups.
	\newblock J. Funct. Anal. \textbf{263}, 2555--2583 (2012)
	
	\bibitem{Paunonen2013SS}
	Paunonen, L.: Robustness of polynomial stability with respect to unbounded
	perturbations.
	\newblock Systems Control Lett. \textbf{62}, 331--337 (2013)
	
	\bibitem{Piskarev2007}
	Piskarev, S., Zwart, H.: {Crank-Nicolson scheme for abstract linear systems}.
	\newblock Numer. Funct. Anal. Optim. \textbf{28}, 717--736 (2007)
	
	\bibitem{Rastogi2020}
	Rastogi, S., Srivastava, S.: Strong and polynomial stability for delay
	semigroups.
	\newblock J. Evol. Equ. \textbf{21}, 441--472 (2021)
	
	\bibitem{Rozendaal2019}
	Rozendaal, J., Seifert, D., Stahn, R.: {Optimal rates of decay for operator
		semigroups on Hilbert spaces}.
	\newblock Adv. Math. \textbf{346}, 359--388 (2019)
	
	\bibitem{Seifert2015}
	Seifert, D.: {A quantified Tauberian theorem for sequences}.
	\newblock Stud. Math. \textbf{227}, 183--192 (2015)
	
	\bibitem{Seifert2016}
	Seifert, D.: {Rates of decay in the classical Katznelson-Tzafriri theorem}.
	\newblock J. Anal. Math. \textbf{130}, 329--354 (2016)
	
	\bibitem{Shi2000}
	Shi, D.H., Feng, D.X.: {Characteristic conditions of the generation of $C_0$
		semigroups in a Hilbert space}.
	\newblock J. Math. Anal. Appl. \textbf{247}, 356--376 (2000)
	
	\bibitem{Staffans2005}
	Staffans, O.J.: Well-Posed Linear Systems.
	\newblock Cambridge, UK: Cambridge Univ. Press (2005)
	
	\bibitem{Tomilov2001}
	Tomilov, Y.: A resolvent approach to stability of operator semigroups.
	\newblock J. Operator Theory \textbf{46}, 63--98 (2001)
	
	\bibitem{Weidmann1980}
	Weidmann, J.: Linear operators in Hilbert spaces.
	\newblock New York: Springer (1980)
	
\end{thebibliography}
\end{document}